\documentclass{amsart}
\usepackage[dvips]{graphicx}
\usepackage{graphicx}
\usepackage{amscd}
\usepackage{amsmath}
\usepackage{amsxtra}
\usepackage{amsfonts}
\usepackage{amssymb}
\usepackage{enumerate}
\usepackage{amsthm}

\newtheorem{theorem}{Theorem}[section]
\newtheorem{corollary}[theorem]{Corollary}
\newtheorem{lemma}[theorem]{Lemma}
\newtheorem{proposition}[theorem]{Proposition}
\theoremstyle{definition}
\newtheorem{definition}[theorem]{Definition}
\newtheorem{remark}[theorem]{Remark}

\theoremstyle{remark}

\renewcommand{\theclaim}{\textup{\theclaim}}

\newtheorem*{acknowledgements}{Acknowledgements}

\numberwithin{equation}{section}

\def\openone

{\mathchoice

{\hbox{\upshape \small1\kern-3.3pt\normalsize1}}

{\hbox{\upshape \small1\kern-3.3pt\normalsize1}}

{\hbox{\upshape \tiny1\kern-2.3pt\SMALL1}}

{\hbox{\upshape \Tiny1\kern-2pt\tiny1}}}

\makeatletter

\newbox\ipbox

\newcommand{\diracb}[1]{\left\langle #1\mathrel{\mathchoice

{\setbox\ipbox=\hbox{$\displaystyle \left\langle\mathstrut
#1\right.$}

\vrule height\ht\ipbox width0.25pt depth\dp\ipbox}

{\setbox\ipbox=\hbox{$\textstyle \left\langle\mathstrut
#1\right.$}

\vrule height\ht\ipbox width0.25pt depth\dp\ipbox}

{\setbox\ipbox=\hbox{$\scriptstyle \left\langle\mathstrut
#1\right.$}

\vrule height\ht\ipbox width0.25pt depth\dp\ipbox}

{\setbox\ipbox=\hbox{$\scriptscriptstyle \left\langle\mathstrut
#1\right.$}

\vrule height\ht\ipbox width0.25pt depth\dp\ipbox}

}\right. }

\newcommand{\dirack}[1]{\left. \mathrel{\mathchoice

{\setbox\ipbox=\hbox{$\displaystyle \left.\mathstrut
#1\right\rangle$}

\vrule height\ht\ipbox width0.25pt depth\dp\ipbox}

{\setbox\ipbox=\hbox{$\textstyle \left.\mathstrut
#1\right\rangle$}

\vrule height\ht\ipbox width0.25pt depth\dp\ipbox}

{\setbox\ipbox=\hbox{$\scriptstyle \left.\mathstrut
#1\right\rangle$}

\vrule height\ht\ipbox width0.25pt depth\dp\ipbox}

{\setbox\ipbox=\hbox{$\scriptscriptstyle \left.\mathstrut
#1\right\rangle$}

\vrule height\ht\ipbox width0.25pt depth\dp\ipbox}

} #1\right\rangle}

\usepackage{graphicx}

\def\blfootnote{\xdef\@thefnmark{}\@footnotetext}

\begin{document}
\title[INFINITE-DIMENSIONAL MEASURE SPACES AND FRAME ANALYSIS.]
{INFINITE-DIMENSIONAL MEASURE SPACES AND FRAME ANALYSIS.}
\author{Palle E.T. Jorgensen}
\address[Palle E.T. Jorgensen]{Department of Mathematics\\
The University of Iowa\\
14 MacLean Hall\\
Iowa City, IA 52242}
\email{jorgen@math.uiowa.edu}

\author{Myung-Sin Song}
\address[Myung-Sin Song]{Department of Mathematics and Statistics\\
Southern Illinois University Edwardsville\\
Box 1653, Science Building\\
Edwardsville, IL 62026}
\email{msong@siue.edu}\ 

\thanks{}
\begin{abstract}
We study certain infinite-dimensional probability measures in connection with frame 
analysis. Earlier work on frame-measures has so far focused on the case of 
finite-dimensional frames. We point out that there are good reasons for a sharp 
distinction between stochastic analysis involving frames in finite vs infinite 
dimensions. For the case of infinite-dimensional Hilbert space $\mathcal{H}$, we study 
three cases of measures. We first show that, for $\mathcal{H}$ infinite dimensional, 
one must resort to infinite dimensional measure spaces which properly contain 
$\mathcal{H}$. The three cases we consider are: (i) Gaussian frame measures, 
(ii) Markov path-space measures, and (iii) determinantal measures.  
\end{abstract}

\maketitle \tableofcontents
\section{Introduction and Setting.}
\label{sec:1}

Over the past two decades, frames have proved to be powerful tools in signal 
processing for a number of reasons, especially on account of their resilience to 
additive noise, to quantization; and because of their numerical stability in their 
use in the reconstruction step, they have improved our ability to capture significant 
signal characteristics. Frame theory is now a dynamic  subject with applications that 
include variety of areas in both mathematics and engineering: operator theory, 
harmonic analysis, wavelet theory, sampling theory, nonlinear sparse approximation, 
wireless communication, data transmission with erasures, filter banks, signal 
processing, image processing, geophysics, quantum computing, sensor networks, and 
more.

In a host of applications, starting with signal and image processing, one makes use of 
inner-product spaces (in our case, choices of suitable Hilbert spaces) in achieving 
efficient signal-representations.
In general, orthogonal expansions are not available. Nonetheless, in many signal 
processing problems,  it is still possible to find overcomplete basis expansions, 
called frame expansions, see the references cited below. For example, in 
analysis/synthesis problems, when
we sample an analog signal above the Nyquist rate, the amplitudes will be coefficients 
in a suitable frame expansion. Such decompositions have received extensive attention 
in the literature when the decomposition parameter is assumed discrete; but, 
nonetheless, the frame-decompositions take a probabilistic form which we shall 
emphasize below. More generally, staying with the overcomplete framework, we present 
here instead a continuous, and a more versatile, probability space approach to these 
expansion problems.  Our applications include determinantal measures, Gaussian frame 
measures; and, more generally, to the setting of Markov path-space measures.

We study frames in Hilbert space $\mathcal{H}$, i.e., systems of vectors in 
$\mathcal{H}$ which 
allow computable representations of arbitrary vectors in $\mathcal{H}$, (especially 
the case when $\mathcal{H}$ is infinite-dimensional) in a way that is analogous to 
more familiar basis expansions. Frames are also called ``over complete" 
systems, and they generalize the better known orthonormal bases (O.N.B.s). Their 
applications include signal processing and wavelet theory, to name only a few.

Frames (Definition \ref{D:1.1}) are systems of vectors in Hilbert space $\mathcal{H}$  
which allow for ``effective" analysis and reconstruction for vectors in $\mathcal{H}$; 
details below.  As is known from the literature on frames, both pure and applied  (see 
e.g., \cite{Pe15}, \cite{FJMP15}, \cite{AuBe15}, \cite{HoLa15}, \cite{PeHaMo15}, 
\cite{QuHiSh15}), there is a scale of ``basis-like" properties that frames may have: 
in one end of the spectrum, there are the orthonormal bases (O.N.B.s), then Parseval 
frames, and in the other end of the spectrum, there are systems of vectors with 
frame-like properties, but where we may lack one of the two bounds, lower or upper, but 
nonetheless, for some other reason, we may still get analysis and reconstruction 
formulas.

Since the early days of Hilbert space axioms and quantum theory and potential theory, 
probability has always played an important role, for example such tools as balayage; 
but it is not until relatively recently that the role of probability has been studied 
systematically in connection with frame analysis. In our present approach, we have been 
especially inspired by the important paper by Ehler and Okoudjou 
\cite{Eh12, EhOk12, EhOk13};  but 
the work in \cite{Eh12, EhOk12, EhOk13}; and in related papers, has so far focused on the case of 
finite-dimensional frames. As we point out below, there are good reasons for a sharp 
distinction between stochastic analysis involving frames in finite vs infinite 
dimensions. The three cases, of measures in infinite dimensions we shall consider are 
the Gaussian measures, Markov path-space measures, and determinantal measures; but our 
present emphasis will be on certain Gaussian families (section \ref{sec:4}).

In our paper, we shall adopt a general notion of probabilistic frames, referring simply 
to methods in frame theory involving probability and stochastic analysis. By 
``frame analysis" we shall refer to a setup where it is possible to construct the dual 
pair of operators, an \textit{analysis operator}, and an associated \textit{synthesis 
operator}. For technical details, see the next section.

Frames let us formulate a harmonic analysis of practical problems, but, so far, there 
are only few harmonic analysis tools available for the analysis of the frames 
themselves. Nonetheless, there are beginnings to a theory of ``frame measures," but so 
far only covering the case when $\mathcal{H}$ is \textit{finite-dimensional}. Our first 
result (section \ref{sec:3}) shows that, unless one passes to a larger measure space, 
the notions of frame measures in finite dimension simply do not go over to 
infinite-dimensional $\mathcal{H}$. On the other hand, we show (section \ref{sec:4}) that 
there is a way to build ambient measures spaces in such a way that we arrive at a rich 
family of \textit{Gaussian wavelet measures}, covering the case when $\mathcal{H}$ is 
infinite-dimensional.

We also study other families of measures associated with frames in infinite dimensions, 
e.g.  Markov measures, and determinantal measures which seem promising. This endeavor 
takes advantage of the probabilistic features already inherent in the axioms of Hilbert 
space as they were developed in the foundations of quantum theory; i.e., the study of 
transition probability, referring to transition between states, for example states of 
different energy levels in atomic models.

The applied mathematicians who use frames have, so far, only developed very few 
quantitative gauges which will tell us how ``different" two given frames might be; or 
will allow us to make precise ``how much" better one frame is as compared to anyone in 
a set of alternatives.

\subsection{Frame Measures} 
\label{sec:1.1}
To help readers appreciate some key features regarding frame measures in the finite 
dimensional case, and their applications, we review some highpoints from \cite{EhOk13}.

In \cite{Eh12, EhOk12, EhOk13}; finite frames in $\mathbb{R}^{N}$ are considered, where frame vectors 
are viewed as discrete mass distributions on $\mathbb{R}^{N}$, the frame concepts 
are extended to probability measures, and the properties of probabilistic frames are 
summarized.  Let $\mathcal{P}:=\mathcal{P}(\mathcal{B}, \mathbb{R}^{N})$ denote the 
collection of \textit{probability measures} on $\mathbb{R}^{N}$ with respect to the 
Borel $\sigma-$algebra $\mathcal{B}$.  The support of $\mu \in \mathcal{P}$, denoted by 
$supp(\mu)$, is the sent of all $x \in \mathbb{R}^{N}$ such that for all open 
neighborhoods $U_{x} \subset \mathbb{R}^{N}$ of $x$, we have $\mu(U_{x})>0$.  
Set $\mathcal{P}(K):=\mathcal{P}(\mathcal{B},K)$ for those probability measures in 
$\mathcal{P}$ whose support is contained in $K \subset \mathbb{R}^{N}$.  The linear 
span of $supp(\mu)$ in $\mathbb{R}^{N}$ is denoted by $E_{\mu}$.
\begin{definition}
\label{D:1.0.1}
\cite{EhOk13} A Borel probability measure $\mu \in \mathcal{P}$ is a 
\textit{probabilistic frame} if there exists $0<A\leq B <\infty$ such that 
\begin{equation}
\label{eq:1.0.1}
  A\|x\|^{2}\leq\int_{\mathbb{R}^{N}}|\langle x,y \rangle|^{2}d\mu(y)\leq B\|x\|^{2}, 
  \quad \text{for all } \in \mathbb{R}^{N}.
\end{equation}
The constants $A$ and $B$ are called lower and upper probabilistic frame bounds, 
respectively.  When $A=B$, $\mu$ is called a tight probabilistic frame.  
\end{definition}

Let
\begin{equation}
\label{eq:1.0.2}
  \mathcal{P}_{2}:=\mathcal{P}_{2}(\mathbb{R}^{N})=
  \{\mu\in \mathcal{P}:M_{2}^{2}(\mu):=\int_{\mathbb{R}^{N}}\|x\|^{2}d\mu(x)<\infty\} 
\end{equation}
be the (convex) set of all probability measures with finite second moments. Frame 
measures in $\mathbb{R}^{N}$ are in $\mathcal{P}_{2}$, and they satisfy 
$E_{\mu}=\mathbb{R}^{N}$.  There 
exists a natural metric on $\mathcal{P}_{2}$ called the $2-$\textit{Wasserstein metric}, 
which is given by
\begin{equation}
\label{eq:1.0.3}
  W_{2}^{2}(\mu, \nu):=\min \{\int_{\mathbb{R}^{N}\times\mathbb{R}^{N}} 
  \|x-y\|^{2}d\gamma(x,y), \gamma \in \Gamma(\mu, \nu)\},
\end{equation}
where $\Gamma(\mu, \nu)$ is the set of all Borel probability measures $\gamma$ on
$\mathbb{R}^{N}\times\mathbb{R}^{N}$ whose marginals are $\mu$ and $\nu$, 
respectively, i.e., $\gamma(A\times \mathbb{R}^{N})=\mu(A)$ and 
$\gamma(\mathbb{R}^{N}\times B)=\nu(B)$ for all Borel subsets $A$, $B$ in 
$\mathbb{R}^{N}$.  \cite{EhOk13} 

Let $\mu \in \mathcal{P}$ be a probabilistic frame.  The probabilisitic 
\textit{analysis operator} is given by
\begin{equation}
\label{eq:1.0.4}
  T_{\mu}:\mathbb{R}^{N}\rightarrow L^{2}(\mathbb{R}^{N}, \mu), \quad 
  x \mapsto \langle x, \cdot \rangle. 
\end{equation}
Its adjoint operator is defined by
\begin{equation}
\label{eq:1.0.5}
  T_{\mu}^{*}:L^{2}(\mathbb{R}^{N}, \mu) \rightarrow \mathbb{R}^{N}, \quad 
  f \mapsto \int_{\mathbb{R}^{N}}f(x)xd\mu(x)   
\end{equation}
and is called the probabilistic \textit{synthesis operator}. The probabilistic Gramiam 
operator of $\mu$ is $G_{\mu}=T_{\mu}T_{\mu}^{*}$.  \cite{EhOk13}  The probabilistic frame operator of 
$\mu$ is $S_{\mu}=T_{\mu}^{*}T_{\mu}$, 
\[
  S_{\mu}:\mathbb{R}^{N} \rightarrow \mathbb{R}^{N}, \quad 
  S_{\mu}(x)=\int_{\mathbb{R}^{N}}\langle x,y \rangle yd\mu(y).
\]
The Gramian of $\mu$, $G_{\mu}$ is the integral operator defined on 
$L^{2}(\mathbb{R}^{N}, \mu)$ by
\begin{equation}
\label{eq:1.0.6}
  G_{\mu}f(x)=T_{\mu}T_{\mu}^{*}f(x)=\int_{\mathbb{R}^{N}}K(x,y)f(y)d\mu(y) 
  =\int_{\mathbb{R}^{N}}\langle x,y \rangle f(y)d\mu(y). 
\end{equation}
Note that $\mathcal{H}=\mathbb{R}^{N}$ and $N<\infty$ in \cite{EhOk13}.
In the next section we extend these tools to infinite dimensions, pointing out a 
number of subtleties, and differences between the two cases, finite vs infinite.

\subsection{Infinite Dimensions} 
\label{sec:1.2}
If $\mathcal{H}$ is an \textit{infinite dimensional} Hilbert space, we shall show that 
the formulas (\ref{eq:1.0.2}), (\ref{eq:1.0.3}), (\ref{eq:1.0.4}) and (\ref{eq:1.0.5})
carry over from $\mathbb{R}^{N}$, $N< \infty$, to dim$\mathcal{H}=\aleph_{0}$; but it 
will be necessary to create an ambient measure space $(\Omega, \mathcal{F})$ where 
$\Omega$ is a certain vector space containing $\mathcal{H}$.  We will show that in this 
case, the four formulas carry over with the following modifications:
In (\ref{eq:1.0.4}), we show that $y \mapsto \langle x,y \rangle$ extends from 
$\mathcal{H}$ to $\Omega$; and the integral in (\ref{eq:1.0.5}) will then be
\begin{equation}
\label{eq:1.0.7}
  x=\int_{\Omega}\langle x, \omega\widetilde{\rangle} \omega d\mu(\omega)
  \quad \text{in $\mathcal{H}$}
\end{equation}
where $\langle, \widetilde{\rangle}$ refers to this extension.  But appropriate 
generalizations of (\ref{eq:1.0.6}) are much more subtle.

The extension of the results in (\ref{eq:1.0.4}) and (\ref{eq:1.0.5}) will involve 
this $\mathcal{H} \rightarrow \Omega$ extension $\sim$: In (\ref{eq:1.0.4}), we will consider 
an analysis operator $T_{\mu}:\mathcal{H} \rightarrow L^{2}(\Omega, \mu)$,
\[
  \mathcal{H} \ni x \rightarrow \langle x, \cdot \widetilde{\rangle} \quad 
  \text{(on $\Omega$);}
\]
and then (\ref{eq:1.0.5}) will read as follows:
\begin{equation}
\label{eq:1.0.8}
  L^{2}(\Omega, \mu)\ni f \overset{T_{\mu}^{*}}{\longmapsto}\int_{\Omega}f(\omega)\omega
  d\mu(\omega)\in \mathcal{H}.
\end{equation} 
Note that, since $\mathcal{H} \subsetneq \Omega$, it is a non-trivial assertion that 
the RHS in (\ref{eq:1.0.8}) is a vector in $\mathcal{H}$.
We now turn to the technical details.

\begin{definition}
\label{D:1.1}
Let $\mathcal{H}$ be a Hilbert space (over $\mathbb{R}$, but $\mathbb{C}$ will work 
also with small modifications).  Let $\alpha$, $\beta \in \mathbb{R}_{+}$, 
$0<\alpha \leq \beta <\infty$.

Set
\begin{equation}
\label{eq:1.1}
  F(\alpha, \beta):=\{ \{\varphi_{n}\}_{n \in \mathbb{N}}; \text{ } \alpha \|x\|^{2}\leq 
  \sum_{n}|\langle x, \varphi_{n} \rangle|^{2} \leq \beta \|x\|^{2}, \forall x\in
  \mathcal{H} \}.
\end{equation}

Let $(\Omega, \mathcal{F})$ be a measure space, $\Omega$ a set, $\mathcal{F}$ a 
$\sigma-$algebra.

We assume further that $\Omega$ is a vector space equipped with a weak$^{*}$-topology 
such that the dual $\Omega'$ satisfies
\begin{equation}
\label{eq:1.2}
  \Omega' \subset \mathcal{H} \subset \Omega, \quad \text{and}
\end{equation}
the inclusion mappings in (\ref{eq:1.2}) are assumed continuous with respect to the 
respective topologies; and $\Omega'$ is dense in $\mathcal{H}$. Equation (\ref{eq:1.2}) 
is an example of a \textit{Gelfand triple}.  Hence, for all $x \in \mathcal{H}$, 
$\langle x, \cdot \rangle$ on $\mathcal{H}$, extends uniquely to a measurable 
function $\langle x,\cdot \widetilde{\rangle}$ on $\Omega$.

Set 
\begin{equation}
\label{eq:1.3}
  FM_{\Omega}(\alpha, \beta):=\{ \text{finite positive measures $\mu$ on 
  $(\Omega, \mathcal{F})$ }; 
\end{equation}
\[
  \alpha \|x\|^{2} \leq \int_{\Omega}|\langle x, 
  \omega \rangle|^{2}d\mu(\omega) \leq \beta \|x\|^{2}, \forall x \in \mathcal{H}  \}.
\]

\end{definition}

\section{Measures Constructed Directly from Frames.}
\label{sec:2}
\subsection{Markov measures from frames.}
\label{sec:2.1}

The purpose of the below is to make the connection between frames with discrete index 
on the one hand, and Markov chains on the other. This in turn allows us to take 
advantage of tools from Markov chains, and to make the connection to continuous 
Markov processes (see section \ref{sec:4} below.)

\begin{proposition}
\label{P:2.1}
Let $\mathcal{H}$ be a Hilbert space, and let $\{\varphi_{n}\}_{n \in \mathbb{N}}$ be 
a frame in $\mathcal{H}$ with frame bounds $\alpha$, $\beta$, 
$0<\alpha\leq\beta<\infty$, i.e.,
\begin{equation}
\label{eq:2.1}
  \alpha\|x\|^{2}\leq \sum_{n\in \mathbb{N}}|\langle x, \varphi_{n}\rangle|^{2}
  \leq \beta \|x\|^{2}
\end{equation}
holds for all $x \in \mathcal{H}$.

We then get a system of transition probabilities
\begin{equation}
\label{eq:2.2}
  p_{x,y}=\frac{|\langle x,y \rangle|^{2}}{c(x)}
\end{equation}
with
\begin{equation}
\label{eq:2.3}
  c(x):=\sum_{n\in \mathbb{N}}|\langle x, \varphi_{n}\rangle|^{2},
\end{equation}
having the following properties: For $x$, $y \in \mathcal{H}\backslash\{0\}$, we have:
\begin{enumerate} [(i)]
\item Reversible:
\[
  c(x)p_{x,y}=c(y)p_{y,x}
\]
\item Markov-rules: $p_{x,\varphi_{n}}\leq 1$ for $n \in \mathbb{N}$,
\[
  \sum_{n \in \mathbb{N}}p_{x,\varphi_{n}}=1.
\]
\item Normalization:
\[
  p_{x,y}\leq \|y\|^{2}/\alpha
\]
where $\alpha$ is the lower frame bound from (\ref{eq:2.1}).
\end{enumerate}
\end{proposition}
\begin{proof}
Rule (i) is immediate from the definition in (\ref{eq:2.2}).  It is also clear from 
(\ref{eq:2.3}) that $p_{x, \varphi_{n}}\leq 1$, for all $n \in \mathbb{N}$.  As for 
(ii), we have:
\[
  \sum_{n}p_{x, \varphi_{n}}=\sum_{n}\frac{|\langle x, \varphi_{n}\rangle|^{2}}{c(x)}
  =\sum_{n}\frac{|\langle x, \varphi_{n}\rangle|^{2}}{\sum_{k}|\langle x, 
    \varphi_{k}\rangle|^{2}}=1,
\]
which is the desired property.

The last property (iii) follows from Schwarz and the lower frame bound as follows:
\[
  p_{x,y}=\frac{|\langle x, y\rangle|^{2}}{c(x)} \leq 
  \frac{\|x\|^{2}\|y\|^{2}}{\alpha\|x\|^{2}}=\|y\|^{2}/ \alpha.
\]
\end{proof}

We now give the path space measures $\mathbb{P}_{x}$:
\begin{corollary}
\label{C:2.1}
Every frame $\{\varphi_{n}\}_{n \in \mathbb{N}}$ defines a Markov process 
$\{X_{k}\}_{k \in \mathbb{N}_{0}}$ as follows:
\begin{equation}
\label{eq:2.4}
  \mathbb{P}_{x}\left(\{\omega:X_{1}(\omega)=n_{1}, \cdots, X_{k}(\omega)=n_{k}\}\right)
  =p_{x, \varphi_{n_{1}}}p_{\varphi_{n_{1}}, \varphi_{n_{2}}}\cdots 
  p_{\varphi_{n_{k-1}}, \varphi_{n_{k}}}.
\end{equation}
\end{corollary}
From the proposition, it follows that (\ref{eq:2.3}) defines a consistent system of 
Markov transitions.  Existence of the corresponding Markov process rule
follows from Kolmogorov's theorem.

\subsection{Determinantal measures from frames.}
\label{sec:2.2}

Starting with a frame $\mathcal{F}$, we arrive at an associated Grammian. Hence for 
each finite subset of $\mathcal{F}$, we get a finite Grammian, and its determinant in 
non-negative, and it induces an $n$ associated determinantal measure. The relevance of 
these measures is discussed below, as well as the continuous-index analogues.

Given a Hilbert space $\mathcal{H}$, dim$\mathcal{H}=\aleph_{0}$.  Let 
$\{\varphi_{n}\}_{n \in \mathbb{N}} \subset \mathcal{H}$ be a system of vectors in 
$\mathcal{H}$ such that the following \textit{a priori} estimate holds:
\begin{equation}
\label{eq:2.5}
  \exists \beta < \infty \quad \text{such that} \quad
  \sum\sum c_{n}c_{m}\langle \varphi_{n}, \varphi_{m}\rangle \leq \beta 
  \sum_{n}|c_{n}|^{2}.
\end{equation}

\begin{remark}
\label{R:2.1}
The estimate (\ref{eq:2.5}) is known to be implied by the upper bound estimate 
(\ref{eq:1.0.1}) in Definition \ref{D:1.0.1}.  When (\ref{eq:2.5}) holds, we talk 
about a Riesz basis sequence. Also, see \cite{Ly03}.
\end{remark}

We shall make use of the following ideas from the setting of determinantal measures, 
see e.g., \cite{Buf16}. More generally, a determinantal point process is a stochastic 
point process, where the local the probability distributions may be represented by 
determinants of suitable kernel functions. In our case, we shall consider the case 
where the local determinants are computed from Grammians  computed from overcomplete 
frame systems (details below). Determinantal processes arise as important tools in 
random matrix theory, in combinatorics, and in physics, see e.g., \cite{GoOl15, OlGr11}.

\begin{proof}
Suppose the upper bound in (\ref{eq:1.0.1}) holds for some $\beta < \infty$, and some 
system $\{\varphi_{n}\}_{n\in \mathbb{N}} \subset \mathcal{H}$.  Then set 
$T:\mathcal{H}\rightarrow l^{2}$,
\begin{equation}
\label{eq:2.6}
  Tx=(\langle x, \varphi_{n} \rangle_{n \in \mathbb{N}}), \quad x \in \mathcal{H}.
\end{equation}
The upper bound in (\ref{eq:1.0.1}) is then equivalent to the following estimate in the
ordering of Hermitian operators
\begin{equation}
\label{eq:2.7}
  T^{*}T \leq \beta I_{\mathcal{H}};
\end{equation}
and so $\|T^{*}T\|\leq \beta$.  It follows that
\[
  \|T\|^{2}=\|T^{*}\|^{2}=\|T^{*}T\|\leq \beta,
\]
and therefore $\|T^{*}\|\leq \sqrt{\beta}$.

But, $c=(c_{n})\in l^{2}(\mathbb{N})$, then $T^{*}c=\sum_{n}c_{n}\varphi_{n}$, and 
\[
  \|T^{*}c\|_{\mathcal{H}}^{2}=\sum_{n}\sum_{m}c_{n}c_{m}
  \langle \varphi_{n}, \varphi_{m} \rangle_{\mathcal{H}} \leq \beta \|c\|_{l^{2}}^{2}
\]
which is the desired estimate (\ref{eq:2.5}). 
\end{proof}

Then note that
\[
  det(\langle \varphi_{j},\varphi_{k} \rangle)_{j,k=1}^{n} \geq 0 \quad 
  \text{for all $n$.}
\]
Then there is a measure $\mu=\mu^{(\varphi)}$ defined on a point configurations in 
$\mathbb{N}$ as follows: Let $\Phi$ be a random point configuration in $\mathbb{N}$, 
then $\mu=\mu^{(\varphi)}$ is determined to be 
\[
  \mu(\Phi \supset \{1,2, \cdots, n\})
  =det(\langle \varphi_{j}, \varphi_{k} \rangle_{j,k=1}^{n})
\]
This measure is called the associated \textit{determinantal measure}.

\section{A Negative Result.}
\label{sec:3}

\begin{theorem}
\label{T:3.1}
Let dim$\mathcal{H}=\aleph_{0}$, and given $\alpha >0$, $\beta < \infty$, then there is 
no Borel measure $\mu$ on $\mathcal{H}$ satisfying
\begin{equation}
\label{eq:3.1}
  \alpha \|x\|^{2} \leq \int_{\mathcal{H}}|\langle x,y \rangle|^{2}d\mu(y)
  \leq \beta \|x\|^{2}, 
\end{equation}
in other words, $FM_{\mathcal{H}}(\alpha, \beta)=\emptyset$.  Also, see \cite{GiSk74}.

\end{theorem}
\begin{proof}
Indirect.  Suppose some finite positive Borel measure $\mu$ exists and satisfies the 
condition (\ref{eq:3.1}) for $\alpha, \beta$ fixed.  Pick an O.N.B. (orthonormal bases) 
$b_{1}, b_{2}, \cdots$ in $\mathcal{H}$, then
\[
  \sum_{n}|\langle x, b_{n} \rangle|^{2}=\|x\|^{2} \quad \text{by Parseval,}
\]
so $\lim_{n\to \infty}\langle x, b_{n} \rangle=0$, pointwise for all $x\in \mathcal{H}$.
Consequently,
\[
  \langle b_{n}, x \rangle^{2} \rightarrow 0 \quad n \rightarrow \infty
\]
\[
  \alpha \leq \int_{\mathcal{H}}|\langle b_{n},y \rangle|^{2}d\mu(y) 
  \leq \beta.
\]
\[
  \int \langle b_{n}, \cdot \rangle^{2} \rightarrow 0 \quad 
  \text{and domination holds.}  
\]  

In summary, the sequence of functions on $\mathcal{H}$, 
$\langle \cdot, b_{n} \rangle \rightarrow 0$ 
as $n\to \infty$ pointwise convergence, and we have the domination, since
\[
  \int|\langle \cdot, b_{n} \rangle|^{2}d\mu(\cdot)\leq \beta \quad \forall n.
\]
So by the Lebesgue dominated convergence theorem,
\[
  \lim_{n\to \infty}\int_{\mathcal{H}}|\langle \cdot, b_{n} \rangle|^{2}d\mu(\cdot)=0
\]
contradicting the lower bound $0<\alpha\leq\int_{\mathcal{H}}|\langle \cdot, b_{n} 
\rangle|^{2}d\mu(\cdot)$ in (\ref{eq:3.1}).  Since $0 < \alpha \leq 0$, we get a 
contradiction.  We have proved that $FM_{\mathcal{H}}(\alpha, \beta)=\emptyset$, 
whenever $0<\alpha\leq\beta<\infty$.

\end{proof}

In \cite{EhOk13}, suppose $N<\infty$,  $\mathcal{H}=\mathbb{R}^{N}$, the authors study 
$\mu \in FM(\alpha, \beta)$
\[
  \alpha \|x\|^{2} \leq \int_{\mathcal{H}}|\langle x,y \rangle|^{2}d\mu(y) 
  \leq \beta \|x\|^{2}, \quad \alpha >0, \quad \beta<\infty.
\]
The theorem shows that new techniques are needed when dim$\mathcal{H}=\infty$.

\section{Gaussian Frame Measures. }
\label{sec:4}

Starting with a separable Hilbert space $\mathcal{H}$, we shall need an associated 
framework from the construction of Gaussian probability measures. We shall then 
discuss how from this we get associated families of probability-frames. Starting with 
$\mathcal{H}$, we first show in Lemma \ref{L:4.1}, that there is a triple of containments 
(see Definition \ref{D:4.1}), with $\mathcal{H}$ contained in $S'$ and a Gaussian probability 
space where the events is the sigma-algebra of subsets of $S'$ generated by the 
cylinder sets. This in turn is based on an application of Minlos' theorem, see also 
section 5 below, and \cite{HiSi08, LuZu12, OkBe08, TlTa15}.

\begin{remark}
\label{R:4.1}
Let $\mathcal{H}$ be a Hilbert space, and assume dim$\mathcal{H}=\aleph_{0}$.  There 
exist $S, S'$ where $S$ is a Fr\'{e}chet space, $S'$ is the dual space of 
$S$ such that $S \subset \mathcal{H} \subset S'$, continuous inclusions, and a Gaussian 
measure $\mu$ on $S'$ such that $\mu \in FM_{S'}(1,1)$.  We can take $\mu$ to be 
Gaussian.  
\end{remark}


The spaces $S$ and $S'$ are as follows.
$S$ is the space of sequences $x= (x_{n})$ which fall off at infinity faster than any 
polynomial in $n$. A sequence $y = (y_{n})$ is in $S'$ if and only if there is a 
positive $M$ so that $(y_{n})$ grows at most like $\mathcal{O}(n^{M})$. We identify 
a system of seminorms on $S$ which turns it into a Fr\'{e}chet space.
The space of continuous linear functionals on $S$ will then coincide with $S'$. 

\begin{definition}
\label{D:4.1}
The spaces $S$ and $S'$.  Both $S$ and its dual $S'$ are sequence spaces 
$x=(x_{n})_{n \in \mathbb{N}}$, $y=(y_{n})_{n \in \mathbb{N}}$; indexed by $\mathbb{N}$
or by $\mathbb{Z}$, and we have
\[
  x \in S \underset{Def}{\iff}\forall k \in \mathbb{N}, \quad \exists C_{k}, 
  \text{ such that } |n|^{k}|x_{n}|\leq C_{k}, \quad \forall n \in \mathbb{N}.
\]
\begin{equation}
\label{eq:4.1}
  y \in S' \underset{Def}{\iff}\exists M \in \mathbb{N}, \exists C< \infty 
  \text{ such that } |y_{n}|\leq C(1+|n|^{M}), \quad \forall n \in \mathbb{N}.
\end{equation}
With the seminorms
\[
  |x|_{k}=\sup_{n}|n|^{k}|x_{n}|,
\]
we note that $S$ becomes a Fr\'{e}chet space, and its dual is $S'$.  We have 
\begin{equation}
\label{eq:4.2}
  S \subset l^{2}(\mathbb{N}) \subset S'.
\end{equation}
\end{definition}

\begin{lemma}
\label{L:4.1}
If $\mathcal{H}$ is a fixed Hilbert space, we pick an O.N.B. $\{b_{n}\}_{n\in \mathbb{N}}$ 
and set $x=\sum_{n}x_{n}b_{n}$, $x_{n}=\langle x,b_{n} \rangle_{\mathcal{H}}$, and 
via the isomorphism $\mathcal{H} \longleftrightarrow l^{2}(\mathbb{N})$, we get 
\[
  S \subset \mathcal{H} \subset S'
\]
where
\[
  x=(x_{n}) \quad x_{n}=\langle x, b_{n} \rangle \quad x\in S, \quad y\in S', \quad 
  \langle x,y \rangle=\sum_{n} x_{n}y_{n}.
\]
\end{lemma}
\begin{proof}
Note that if $y\in S'$ satisfies (\ref{eq:4.1}) for some $M$, then 
\[
  |\langle x,y \rangle|\leq Const\text{ } |x|_{M+2} \quad \text{for all $x \in S$.}
\]

Since $S$ is dense in $l^{2}(\mathbb{N})$, the ``inclusion" 
$l^{2}(\mathbb{N}) \hookrightarrow S'$ is indeed $1-1$.
\end{proof}

The topology, and the $\sigma-$algebra, on $S'$ is generated by the following subsets 
of $S'$, the cylinder-sets.  They are indexed by $k \in \mathbb{N}$, open subsets 
$\mathcal{O} \subset \mathbb{R}^{k}$, and subsets $\{x_{i}\}_{i=1}^{k}\subset S$ with
\begin{equation}
\label{eq:4.3}
  Cyl(\{x_{i}\},\mathcal{O})=\{ \omega \in S'; \text{ } 
  (\langle x_{i}, \omega \rangle)_{1}^{k} \in \mathcal{O}\}.
\end{equation}

The cylinder-sets form a basis for both a topology on $S'$ (making it the dual of $S$), and 
of a $\sigma-$algebra.  We shall use both.

We now verify why we need the space $S'$ with $\mu$ a positive measure  
defined on the cylinder Borel $\sigma-$algebra of subsets of $S'$. 

Let dim$\mathcal{H}=\infty$ and $(b_{n})$ O.N.B. We have
$S \subset \mathcal{H} \subset S'$.  
We use the $\sigma-$algebra subsets of $S'$ generated by the cylinder sets.

We shall adopt the following standard terminology from probability theory: By a 
probability space we mean a triple $(\Omega, \mathcal{F}, \mu)$ , i.e., sample space, 
sigma-algebra, and probability measure. The $\mathcal{F}$-measurable functions $f$ on 
$\Omega$  are the random variables, The integral of $f$ with respect with $\mu$ is 
called the expectation, and it is denoted $\mathbb{E}(f)$.

Now, for the Gaussian measures: There exists a measure $\mu$, Gaussian on $S'$ with 
$\Omega = S'$, we have

\begin{equation}
\label{eq:4.3.1}
  \mathbb{E}(f):=\mathbb{E}(\Omega, \mathcal{F}, \mu, f)=\int_{\Omega} f d\mu, \quad 
  \int_{y\in S'}|\langle x,y \rangle|^{2} d\mu(y)=\|x\|^{2}
\end{equation}
or $\mathbb{E}(|\langle x, \cdot\rangle|^{2})=\|x\|^{2}$. See also Lemma \ref{L:5.1}.
More generally, we have the following theorem:

\begin{theorem}
\label{T:4.1}
\cite{TlTa15}, \cite{OkBe08}, \cite{HiSi08} (Minlos' theorem) There exists a unique Gaussian measure $\mu$ on $S'$ such that
\begin{equation}
\label{eq:4.3.2}
  \mathbb{E}(e^{i\langle x, \cdot \rangle})=e^{-\frac{1}{2}\|x\|^{2}} \quad 
    \text{holds for all $x \in \mathcal{H}$.} 
\end{equation}     
Where (\ref{eq:4.3.1}) is applied to 
$\omega \rightarrow e^{i\langle x, \omega \rangle}$
on the LHS in (\ref{eq:4.3.2}). 
The RHS is called a Gaussian covariance function.
\end{theorem}

Consider $S \subset \mathcal{H} \subset S'$, $S$ is a Fr\'{e}chet space with a nuclear 
embedding, and a Gelfand triple.  (See \cite{Jo14}, \cite{LuZu12}, 
\cite{JoPe11}.)
\[
  \mathbb{E}(\langle x, \cdot \rangle^{2k+1})=\int \langle x, \cdot \rangle^{2k+1}d\mu=0,
\]
\[
  \mathbb{E}(\langle x, \cdot \rangle^{2k})=\int \langle x, \cdot \rangle^{2k}d\mu
    =(2k-1)!!\|x\|^{2k},
\]
where $(2k-1)!!=\frac{(2k)!}{2^{k}\cdot k!}=1\cdot 3 \cdot 5 \cdots (2k-1)$, 
starting with
\[
  \int \langle x, \cdot \rangle d\mu =0.
\]
Note that since $\mu$ is Gaussian, it is determined by its first two moments.


We now turn to the Gaussian process associated with a fixed frame:
\begin{corollary}
\label{C:4.3}
Let $\{\varphi_{n}\}_{n \in \mathbb{N}}$ be a fixed frame in $\mathcal{H}$, 
dim$\mathcal{H}=\aleph_{0}$; see Definition \ref{D:1.1}.  Let $\mu$ denote the 
Gaussian measure in Theorem \ref{T:4.1}; and let 
$T_{\mu}:\mathcal{H}\rightarrow L^{2}(S',\mu)$ be the canonical isometry 
$T_{\mu}x=\langle x,\cdot \widetilde{\rangle}$, $x\in \mathcal{H}$. Then the Gaussian 
covariance matrix for the Gaussian process $\{T_{\mu}\varphi_{k}\}_{k\in \mathbb{N}}$
is $(\langle \varphi_{k}, \varphi_{n}\rangle_{\mathcal{H}})$, i.e., the Gramian of the 
frame.
\end{corollary}
\begin{proof}
Given a fixed frame $\{\varphi_{n}\}_{n \in \mathbb{N}}$, with frame constants $\alpha$, 
$\beta \in \mathbb{R}_{+}$, $0<\alpha \leq \beta <\infty$.  Set 
$G=(\langle \varphi_{j},\varphi_{k} \rangle)_{\mathbb{N} \times \mathbb{N}}$, and 
$G_{n}=(\langle \varphi_{j}, \varphi_{k} \rangle)_{j,k=1}^{n}$, $n \times n$ matrix. 
Let $\mu$ be the Gaussian measure from Theorem \ref{T:4.1}.  Then, for every 
$n \in \mathbb{N}$, the joint distribution of the system 
$\{T_{\mu}\varphi_{k}\}_{k=1}^{n}$ of Gaussian random variables is 
\begin{equation}
\label{eq:4.4}
  (detG_{n})^{-\frac{1}{2}}e^{-\frac{1}{2}\sum_{j=1}^{n}\sum_{k=1}^{n}x_{j}x_{k}
  (G_{n}^{-1})_{j,k}}
  \underset{\text{$n-$dimensional Lebesgue measure}}{dx_{1}\cdots dx_{n}}.
\end{equation}
Indeed, a direct computation, using (\ref{eq:4.4}) shows that, when 
$j,k\in\{1,2,\cdots,n\}$, we get
\begin{align*}
  (detG_{n})^{-\frac{1}{2}}\int_{\mathbb{R}^{n}}x_{j}x_{k}
  e^{-\frac{1}{2}\langle x, G_{n}^{-1}x \rangle}dx_{1}\cdots dx_{n}
  &=\langle T_{\mu}\varphi_{j}, T_{\mu}\varphi_{k} \rangle_{L^{2}(\mu)} \\
  &=\langle \varphi_{j}, \varphi_{k} \rangle = G_{j,k}
\end{align*}
which is the desired conclusion.
\end{proof}

\section{Analysis and Synthesis From Gaussian Frame Measures.}
\label{sec:5}

We recall Minlos' theorem.
Construct $S \subset \mathcal{H} \subset S'$.  On the cylinder $\sigma-$algebra of 
subsets of $S'$ with $\mu$ a probability measure. 

A good reference for the present discussion regarding  infinite-dimensional Gaussian 
distributions is \cite{Bo88}.  Nonetheless we have included below enough details in 
order to make our paper readable for a general audience. This reference addresses in 
detail such subtleties as measurability; and the  fact that in the 
infinite-dimensional case, the measure of $\mathcal{H}$ is zero.

\begin{lemma}
\label{L:5.1}
For all $x \in \mathcal{H}$, 
$\langle x, \cdot \rangle$ on $\mathcal{H}$ has an extension to $S'$ (denote also by
$\langle x, \omega \rangle$, $\omega \in S'$ such that 
\begin{equation}
\label{eq:5.1}
  \int_{S'}e^{i\langle x, \omega\rangle}d\mu(\omega)=\mathbb{E}(e^{iT_{\mu}x})=
  e^{-\frac{1}{2}\|x\|^{2}}.
\end{equation}
Also, see \cite{Bo88}.
\end{lemma}

\begin{proof}
By comparing the two power series, we then get
\begin{equation}
\label{eq:5.2}
  \int_{S'}|\langle x, \omega \rangle|^{2}d\mu(\omega)=\|x\|^{2}
\end{equation}
where (\ref{eq:5.2}) or equivalently $T_{\mu}$ is called an Ito-isometry.  
See \cite{All06}, \cite{Kam96}.
\end{proof}

\begin{corollary}
\label{C:5.1}
The inner product $\langle x,\cdot \rangle$ on $\mathcal{H}$ extends to $S'$ as 
follows:
\begin{equation}
\label{eq:5.3}
  T_{\mu}x=\langle x, \cdot \widetilde{\rangle} \quad \text{on $S'$}
\end{equation}
such that $T_{\mu}$ is an isometry, see (\ref{eq:5.2}), from $\mathcal{H}$ into 
$L^{2}(S', \mu)$.
Also, see \cite{Bo88}.
\end{corollary}

The Gaussian property $\mu$ is part of Minlos' theorem.

From the corollary we get the adjoint operator,
\[
  T_{\mu}^{*}:L^{2}(S',\mu)\longrightarrow \mathcal{H}
\]
as a co-isometry.

An important technical point is that the operator mapping $x$ to the measurable extension 
of the specified linear functional. We leave to the reader checking that indeed we get 
agreement between the following two: (i) the infinite-dimensional integration with the 
measurable linear functionals; and  (ii) the finite dimensional inner product from 
Lemma \ref{L:5.2}. This consistency issue is important for Corollaries \ref{C:5.2} and 
\ref{C:5.3}.  

\begin{lemma}
\label{L:5.2}
For $T_{\mu}^{*}$ we have
\begin{equation}
\label{eq:5.4}
  (T_{\mu}^{*}f)=\int_{S'}f(\omega)\omega d\mu(\omega), \quad \forall f \in L^{2}(S', 
      \mu), 
  \quad \text{and}
\end{equation}
RHS in (\ref{eq:5.4}) is in $\mathcal{H}$.
\end{lemma}

Suppose $x\in\mathcal{H}$, then $T_{\mu}x \in L^{2}(\mu)$; and if $f \in L^{2}(\mu)$, 
then $T_{\mu}^{*}f\in\mathcal{H}$.  In summary we have a dual pair of operators:
\begin{equation}
\label{eq:5.5}
  \mathcal{H} \overset{T_{\mu}}{\underset{T_{\mu}^{*}}\rightleftharpoons} 
    L^{2}(S', \mu)
\end{equation}

\begin{proof}
To show that $\int_{S'}f(\omega)\omega d\mu(\omega)\in \mathcal{H}$, we can use Riesz, 
and instead prove the following a priori estimate:  
\begin{equation}
\label{eq:5.6}
  \exists \text{ } Const<\infty, \quad \text{such that} \quad 
  \left\vert \int_{S'}f(\omega)\langle x, \omega \rangle d\mu(\omega) \right\vert^{2} 
  \leq Const \text{ }\|x\|^{2} \quad \forall x \in \mathcal{H}.
\end{equation}
Then we conclude that $\int_{S'}f(\omega)\omega d\mu(\omega)\in \mathcal{H}$.

Details:
\begin{align*}
  LHS (\ref{eq:5.6})&\underset{Schwarz}\leq \int_{S'}|f|^{2}d\mu
  \int_{S'}|\langle x,\omega \rangle|^{2}d\mu(\omega) \\
  &=\|f\|_{L^{2}(\mu)}^{2}\|\langle x, \cdot \rangle \|_{L^{2}(\mu)}^{2} \\
  &\underset{\text{by } (\ref{eq:5.2})}= \|f\|_{L^{2}(\mu)}^{2}\|x\|_{\mathcal{H}}^{2}, 
  \quad \forall x \in \mathcal{H}, \quad C=\|f\|_{L^{2}(\mu)}^{2}
\end{align*}
$T_{\mu}$ is called the Ito-isometry.  See \cite{All06}, \cite{Kam96}.

Apply Riesz to the Hilbert space $\mathcal{H}$, and we conclude that the integral 
below is a vector in $\mathcal{H}$, i.e., that 
\[
  T_{\mu}^{*}f=\int_{S'}f(\omega)\omega d\mu(\omega) \in \mathcal{H}.
\]
\end{proof}

\begin{corollary}
\label{C:5.2}
For every $x \in \mathcal{H}$, the following frame decomposition
\begin{equation}
\label{eq:2.7}
  x=\int_{S'}\langle x,\omega \rangle \omega d\mu(\omega).
\end{equation}
\end{corollary}
\begin{proof}
We showed that $T_{\mu}$ is isometry (see (\ref{eq:5.6})), and that $T_{\mu}^{*}$ is 
given by (\ref{eq:5.4}).  Hence
\begin{equation}
\label{eq:5.8}
  T_{\mu}^{*}T_{\mu}=I_{\mathcal{H}}, \quad \text{so}
\end{equation}
\begin{equation}
\label{eq:5.9}
  x=T_{\mu}^{*}T_{\mu}x.
\end{equation}
We write out the RHS in (\ref{eq:5.9}) as
$\int_{S'}\langle x,\omega \rangle \omega d\mu(\omega )=x$, since 
$\int_{S'}f(\omega)\omega d\mu(\omega)\in \mathcal{H}$ 
if $f\in L^{2}(\mu)$.
\end{proof}

\begin{corollary}
\label{C:5.3}
Let $\mathcal{H}$, $S'$ and $\mu$ be as above, and let $T_{\mu}$ and $T_{\mu}^{*}$ be 
the corresponding operators in Lemma \ref{L:5.1}; then $T_{\mu}T_{\mu}^{*}$ is the 
projection onto the range of $T_{\mu}$; i.e., 
\[
  Q_{\mu}=T_{\mu}T_{\mu}^{*}=proj\{T_{\mu}x; \text{} x \in \mathcal{H}\}.
\]
\end{corollary}
\begin{proof}
By (\ref{eq:5.8}) above, we have
\begin{align*}
  Q_{\mu}^{2}&=(T_{\mu}T_{\mu}^{*})(T_{\mu}T_{\mu}^{*}) \\
             &=T_{\mu}(T_{\mu}^{*}T_{\mu})T_{\mu}^{*} \\
             &\underset{\text{by (\ref{eq:5.8})}}{=}T_{\mu}T_{\mu}^{*}=Q_{\mu}.
\end{align*}

\end{proof}

\begin{definition}
\label{D:5.1}
Let $x \in \mathcal{H}$, and let $\mu$ and $\mu^{x}$ be Gaussian measures, then
\[
  \int_{S'}\varphi d\mu^{x}=\int\varphi(\cdot+x)d\mu(\cdot)
\]
\[
  \mu^{x}(E)=\mu(E-x) \quad \text{where $E \subset S'$.} 
\]

\end{definition}

There are several candidates for frame measures in the case of infinite-dimensional 
separable Hilbert space $\mathcal{H}$, i.e., $\mathcal{H} \simeq l^{2}(\mathbb{N})$, 
one is the case of 

1. \underline{Gaussian measures} $\mu$ supported in a measure space $S'$ derived from a 
Gelfand triple $S \subset \mathcal{H} \subset S'$ where $S$ is a Fr\'{e}chet space, 
$S \hookrightarrow \mathcal{H}$ is continuous, on $S'=$ the dual of $S$.  If $\mu$ is 
determined from 
\[
  \int_{S'}e^{i\langle x,\cdot \rangle}d\mu(\cdot)=e^{-\frac{1}{2}\|x\|^{2}}
\]
then
\[
  \int_{S'}|\langle x,y \rangle|^{2}d\mu(y)=\|x\|^{2}
\]
holds for all $x \in \mathcal{H}$. See \cite{Jo14}, \cite{LuZu12}, 
\cite{JoPe11}.

Given a vector $x$, then the Radon-Nikodym derivative 
\[
  \frac{d\mu^{x}}{d\mu}=e^{(T_{\mu}x)(\omega)-\frac{1}{2}\|x\|^{2}}, \quad 
  \text{will represent a multiplier for an associated Ito-integral; see also 
    (\ref{eq:6.1}) below.}
\]
$\mu(S')=1$, $S \subset \mathcal{H} \subset S'$
$y \longrightarrow \langle x, y \rangle$
\[
  \int_{S'} \tilde{x}d\mu=0
\]
\[
  \int_{S'} |\tilde{x}|^{2}d\mu=\|x\|^{2}, \quad x \in \mathcal{H}.
\]

\section{Translation.}
\label{sec:6}

Now, let $\mathcal{H}$ be such that dim$\mathcal{H}=\aleph_{0}$ and 
$S \subset \mathcal{H} \subset S'$.  Let $\mu$ be Gaussian probability measure in $S'$,
and 
\[
  T_{\mu}:\mathcal{H} \underset{isometry}{\longrightarrow}L^{2}(S',\mu)
\]
\[
  T_{\mu}x=\langle x, \cdot \widetilde{\rangle} \quad x \in \mathcal{H} \quad \text{extension from 
    $\mathcal{H}$ to $S'$.}
\]

Applications of 
\[
  T_{\mu}:\mathcal{H}\rightarrow L^{2}(S',\mu)
\]

\begin{theorem}
\label{T:6.1}
We can define $\mu^{x}$ 
\[
  \mu^{x}(E):=\mu(E-x), \quad x \in \mathcal{H}, \quad E \subset S' \quad \text{Borel}
\]
and the Radon-Nikodym derivative is
\[
  \frac{d\mu^{x}}{d\mu}\in L_{+}^{1}(S',\mu)
\]
\begin{equation}
\label{eq:6.1}
  \frac{d\mu^{x}}{d\mu}(\omega)=e^{(T_{\mu}x)(\omega)-\frac{1}{2}\|x\|^{2}}, \quad 
  \omega \in S'.
\end{equation}
See, \cite{Bo88}.
\end{theorem}

\begin{theorem}
\label{T:6.2}
Let $\mathcal{H}$ and $\mu$ be as above, dim$\mathcal{H}=\aleph_{0}$; and let 
$x,y \in \mathcal{H}$.  Set
\begin{equation}
\label{eq:6.2}
  \mathcal{E}_{\mu}(x)(\cdot)=e^{(T_{\mu}x)(\cdot)}e^{-\frac{1}{2}\|x\|^{2}} \quad 
  \text{on $S'$,}
\end{equation}
see (\ref{eq:6.1}).  Then
\begin{equation}
\label{eq:6.3}
  \int_{S'}\mathcal{E}_{\mu}(x)(\omega)\langle y,\omega \widetilde{\rangle}^{2}d\mu(\omega)
  =\langle x,y \rangle^{2}+\|y\|^{2};
\end{equation}
and the following co-cycle property holds:
\begin{equation}
\label{eq:6.4}
  \mathcal{E}_{\mu}(x_{1})(\omega)\mathcal{E}_{\mu}(x_{2})(\omega)
  =e^{-\langle x_{1},x_{2}\rangle_{\mathcal{H}}}\mathcal{E}_{\mu}(x_{1}+x_{2})(\omega), 
  \quad \text{for all $x_{1}, x_{2}\in\mathcal{H}$, and $\omega \in S'$.}
\end{equation}
\end{theorem}
\begin{proof}
\begin{align*}
  \int_{S'}\mathcal{E}_{\mu}(x)(\omega)\langle y,\omega \widetilde{\rangle}^{2}d\mu(\omega)
  &=\int_{S'}\langle y, x+\omega \rangle^{2} d\mu(\omega) \\
  &=\int_{S'}(\langle y,x \rangle^{2}+\langle y,\omega \widetilde{\rangle}^{2}
    +2\langle y,x\rangle\langle y,\omega \widetilde{\rangle})d\mu(\omega)  \\
  &\underset{\text{by Theorem \ref{T:4.1}}}{=} \langle y,x \rangle^{2} + 
  \int_{S'} \langle y,\omega \widetilde{\rangle}^{2}d\mu(\omega) \\
  &=\langle y,x \rangle^{2}+\|y\|^{2}
\end{align*}
which is the desired conclusion.  The co-cycle property (\ref{eq:6.4}) is immediate 
from (\ref{eq:6.2}).
\end{proof}

\begin{corollary}
\label{C:6.1}
For each Parseval frames $(\varphi_{n})$ in $\mathcal{H}$, there exists an associated 
i.i.d. $N(0,1)$, system $(Z_{n})$ on $L^{2}(S',\mu)$ such that
\[
  (Tx)(\omega)=\sum_{n}\langle x, \varphi_{n}\rangle Z_{n}(\omega), \quad 
  \omega \in S'.
\]
\end{corollary}
\begin{proof}
There exists an isometry $V:\mathcal{H} \rightarrow l^{2}(\mathbb{N})$.  If 
$\epsilon_{k}(n)=\delta_{k,n}$ in $l^{2}(\mathbb{N})$ then 
$V^{*}\epsilon_{k}=\varphi_{k}$, for all $k \in \mathbb{N}$.  See \cite{Jor08}, 
\cite{JoTi15}.
\end{proof}

\underline{Gaussian Karhunen-Lo\`{e}ve Expansion.}  Suppose $\{\varphi_{n}\}$ is a 
Parseval frame.  We then have
\[
  x=\sum_{n}\langle x, \varphi_{n}\rangle \varphi_{n}
\]
\[
  Tx=\sum_{n}\langle x, \varphi_{n}\rangle T\varphi_{n}
  =\sum_{n}\langle x, \varphi_{n}\rangle Z_{n}(\cdot).
\]
By Karhunen-Lo\`{e}ve expansion, there exists an i.i.d. $Z_{n}$, $N(0,1)$ system such that 
$T\varphi_{n}=TV^{*}\epsilon_{n}=Z_{n}(\cdot)$.  Hence
\[
  \|Tx\|_{L^{2}(\mu)}^{2}=\sum_{n}|\langle x, \varphi_{n}\rangle|^{2}=\|x\|^{2}
\]
because $(\varphi_{n})$ is a Parseval frame.

\begin{acknowledgements} The authors are please to acknowledge helpful discussions, 
both recent and not so recent, with Ilwoo Cho, D. Dutkay, 
Keri Kornelson, P. Muhly, Judy Packer, Erin Pearse, 
Steen Pedersen, Gabriel Picioroaga, Karen Shuman. 
\end{acknowledgements}


\end{document}